\documentclass[3p]{elsarticle}
\usepackage{amsmath}
\usepackage{amssymb}
\usepackage{amsbsy}
\usepackage{mathrsfs}
\usepackage{tikz}
\usepackage{pgfplots}
\usepackage{cases}
\usepackage{hyperref} 
\usepackage{comment}
\usepackage{ulem}
\usepackage{subcaption}
\usepackage{float}
\usepackage{soul,cancel}
\def\phi {{\varphi}}

\newtheorem{theorem}{Theorem}
\newtheorem{algorithm}{Algorithm}[section]
\newtheorem{lemma}[theorem]{Lemma}
\newdefinition{remark}{Remark}
\newdefinition{hyp}{Hypothesis}
\newdefinition{ass}{Assumption}
\newproof{proof}{Proof}

\def\V#1{\mathbf{#1}}
\def\Vt#1{\mathbf{\tilde{#1}}}
\newcommand{\Mbd}{\mathcal{M}}
\newcommand{\divergence}{\mathrm{\nabla \cdot}}
\newcommand{\Lmap}{\mathrm{L}}
\newcommand{\Linvmap}{\mathrm{L^{-1}}}

\def\prodL2#1{\left(#1\right)}

\newcommand{\normBoch}[1]{{\left\vert\kern-0.25ex\left\vert\kern-0.25ex\left\vert #1 
\right\vert\kern-0.25ex\right\vert\kern-0.25ex\right\vert}}

\newcommand{\DY}[1]{{\color{black}{#1}}}

\journal{     }

\begin{document}
\begin{frontmatter}

\title{Steady Incremental Viscosity Splitting Method for solving the stationary Navier-Stokes equation} 

\author[labelA]{Aziz Takhirov} 
\ead{atakhirov@sharjah.ac.ae}
\address[labelA]{Department of Mathematics, University of Sharjah, UAE}
\author[labelB]{Driss Yakoubi}
\ead{driss.yakoubi@devinci.fr}
\address[labelB]{De Vinci Higher Education, De Vinci Research Center, Paris, France.}


\begin{abstract}

We develop a novel and efficient iterative scheme for solving incompressible steady Navier-Stokes equations. The method is an adaptation of the Incremental Viscosity Splitting approximation for unsteady flows \cite{Yak2023} to steady equations. At each nonlinear iteration, the scheme requires solving an elliptic PDE for the velocity variable and a system with an SPD matrix for the pressure variable, which remains the same across all nonlinear iterations. The method can also be interpreted as an algebraic splitting approach. We prove boundedness and geometric convergence. Numerical tests illustrate the efficiency of the proposed algorithm.  
\end{abstract}


\end{frontmatter}

 \section{Introduction}\label{Sec-Intro}
 The goal of this paper is to present a novel splitting method for solving the stationary incompressible Navier-Stokes equation
\begin{align}
- \nu\Delta \V{u} 
+ \V{u} \cdot \nabla \V{u}  + \nabla p - \gamma \nabla \divergence \V{u} &= \V{f}  \quad \text{in }  \Omega, \label{eq:NSE1} \\
\nabla \cdot \V{u} &= 0 \quad \text{ in } \Omega, \label{eq:NSE2} \\
\V{u} & = \V{0} \quad \text{on } \partial \Omega,
\label{eq:NSE3}
\end{align}
where $\V{u}$ is the velocity of the fluid, $p$ is the pressure, $\V{f}$ is the body forcing, $\nu$ is the kinematic viscosity. For simplicity, we consider homogeneous Dirichlet boundary conditions for $\V{u}$. The grad-div stabilization term $- \gamma \nabla \divergence \V{u}$, $\gamma \ge 0$, in \eqref{eq:NSE1} is zero for the exact solution, but it will be an important ingredient for constructing our approximation. Numerous advantages of the grad-div stabilization are well-known, cf. \cite{JenJohLinReb2014,OlsReu2004,https://doi.org/10.1002/fld.3654,MOULIN2019718}. 

Solving the system \eqref{eq:NSE1}-\eqref{eq:NSE3} is necessary for moderate Reynolds number flows or time-averaged unsteady Navier-Stokes equations. To fix the ideas, let's consider the standard Picard linearization \cite{GirRav1986} of \eqref{eq:NSE1}-\eqref{eq:NSE3}: 
\begin{align}
- \nu\Delta \V{u}_k 
+ \V{u}_{k-1} \cdot \nabla \V{u}_k + \nabla p_k  &= \V{f}  \quad \text{in } \; \Omega, \label{eq:Picard1} \\
\nabla \cdot \V{u}_k &= 0  \quad \text{ in } \; \Omega, \label{eq:Picard2} \\
\V{u}_k & = \V{0} \quad \text{on } \; \partial \Omega.
\label{eq:Picard3}
\end{align}
Denoting by $\V{U}_k \in \mathbb{R}^n, \V{P}_k \in \mathbb{R}^m$ the coefficient vectors corresponding to the spatial discretizations of $\V{u}_k$ and $p_k$, \eqref{eq:Picard1}-\eqref{eq:Picard3}, requires the solution of the following non-symmetric saddle-point linear system
\begin{equation}\label{eq:PicardSystem1}
\begin{bmatrix} 
A_{k-1} & B^T \\ 
B &   0
\end{bmatrix}
\begin{bmatrix} \V{U}_k  \\ 
\V{P}_k  
\end{bmatrix}
= \begin{bmatrix} \V{F} \\ \V{0} \end{bmatrix},
\end{equation}
at each iteration $k$, where $A_{k-1}$ is the $n \times n$ matrix corresponding to velocity terms in \eqref{eq:Picard1}, and $B$ is $m \times n$ matrix corresponding to the pressure term. Direct solvers usually run out of memory for 3D problems, and solving this saddle-point system iteratively is also notoriously difficult. 

Since the classical work of Chorin \cite{Cho1968}, the numerical approximation of the unsteady counterpart of system \eqref{eq:NSE1}-\eqref{eq:NSE3} that gives easier linear systems to solve has received much attention in the literature  \cite{Gue2009,GueMinShe2006a}. Many of these fractional-type schemes heavily rely on the presence of the discretized time derivative of $\V{u}$. On the other hand, such splitting schemes for genuinely steady problems are less studied, cf. \cite{chen17,TC23,GEREDELI2023114920}. 

One can formally obtain a decoupling solution approach for \eqref{eq:PicardSystem1} by performing the following (incremental) block LU decomposition of \eqref{eq:PicardSystem1} as
\begin{equation}\label{eq:PicardSystem2}
\begin{bmatrix} 
A_{k-1} & 0 \\ 
B &  -B A_{k-1}^{-1} B^T 
\end{bmatrix}
\begin{bmatrix} 
I & A^{-1}_{k-1} B^T \\
0 & I 
\end{bmatrix}
\begin{bmatrix} \V{U}_k  \\ 
\V{P}_k - \V{P}_{k-1} 
\end{bmatrix}
= \begin{bmatrix} \V{F} - B^T \V{P}_{k-1} \\ \V{0} \end{bmatrix}. 
\end{equation}
The system \eqref{eq:PicardSystem2} formally splits the velocity, and the pressure solves. However, the pressure solve needs the inversion of the $m\times m$ Schur complement matrix $S_{k-1}=-B A_{k-1}^{-1} B^T$ at each iteration $k$, rendering the scheme inefficient for large 3D problems.

More effective algebraic splitting methods for steady incompressible Navier–Stokes equations can be derived by observing that \cite{VIGUERIE2018271}, for any two $n \times n$ matrices $H_1, H_2$, an ILU decomposition of the form
\begin{equation}\label{eq:PicardSystem3}
\begin{bmatrix} 
A_{k-1} & 0 \\ 
B &  -B H_1 B^T 
\end{bmatrix}
\begin{bmatrix} 
I & H_2 B^T \\
0 & I 
\end{bmatrix}
\begin{bmatrix} \V{U}_k  \\ 
\V{P}_k - \V{P}_{k-1} 
\end{bmatrix}
= \begin{bmatrix} \V{F} - B^T \V{P}_{k-1} \\ \V{0} \end{bmatrix} 
\end{equation}
is also strongly consistent approximations of \eqref{eq:NSE1}-\eqref{eq:NSE3}, provided they generate stable schemes. Using this observation, the authors of \cite{rebholz2019efficient} considered an iterative Incremental Picard-Yosida scheme based on the following ILU factorization:
\begin{equation}\label{eq:IPY}
\begin{bmatrix} 
A_{k-1} & B^T \\ 
B &   B\, \left(A^{-1}_{k-1} -  \widetilde{A}^{-1}\right) \, B^T
\end{bmatrix}
= 
\begin{bmatrix} 
A_{k-1} & 0 \\ 
B &  -B \widetilde{A}^{-1} B^T 
\end{bmatrix}
\begin{bmatrix} 
I & A^{-1}_{k-1} B^T \\
0 & I 
\end{bmatrix},
\end{equation}
where the $n\times n$ matrix $\widetilde{A}$ corresponds to linear velocity terms in \eqref{eq:NSE1}. 
A related algebraic splitting of \cite{VIGUERIE2018271} is based on separating the nonlinear term into implicit and explicit parts using a weight parameter $\alpha \in [0,1]$:
\begin{equation}\label{eq:PicardSystem4}
\begin{bmatrix} 
\widetilde{A} + \alpha C_{k-1} & B^T \\ 
B &  0 
\end{bmatrix} = 
\begin{bmatrix} 
\widetilde{A} + \alpha C_{k-1} & 0 \\ 
B &  -B \widetilde{A}^{-1} B^T 
\end{bmatrix}
\begin{bmatrix} 
I & \widetilde{A}^{-1} B^T \\
0 & I 
\end{bmatrix},    
\end{equation}
where the right-hand-side vector is 
\begin{equation}
    \label{eq:PicardSystem5}
\begin{bmatrix} \V{F} - B^T \V{P}_{k-1} - (1-\alpha) C_{k-1} \\ \V{0} \end{bmatrix},
\end{equation}
and $C_{k-1}$ is of size $n \times n$ corresponds to convection term. 

In this work, inspired by \cite{Yak2023}, we propose a novel Steady Incremental Viscosity Splitting (SIVS) method with grad-div stabilization. Unlike projection schemes, the unsteady Viscosity Splitting scheme \cite{Yak2023} yields an end-of-step velocity field that is both divergence-free and satisfies the correct boundary conditions. The scheme that we propose in this paper also gives (discretely) div-free velocity subject to full Dirichlet boundary conditions. Moreover, the scheme can be interpreted as an algebraic splitting method akin to \eqref{eq:PicardSystem3}, which requires much simpler linear system solvers than the standard Picard system \eqref{eq:PicardSystem2}. We also mention the related work \cite{10.1063/5.0195184}, which studied a non-incremental viscosity splitting method for a stationary second-grade fluid model without any grad-div stabilizations.

This paper is arranged as follows. In the next Section, we introduce the notations and some preliminary results. In Section \ref{Sec-Algo}, we present our Algorithm, establish uniform boundedness and convergence results, and also present the algebraic splitting viewpoint of our scheme. In Section \ref{Sec-Simulation}, we present numerical tests, and Section \ref{Sec-Conclusion} includes the concluding remark.

\section{Notations and preliminaries}\label{Notations}
Throughout this work, vector fields and spaces of the form $\mathcal{S}^d$ are denoted using boldface notation.
\noindent Standard notations for Sobolev spaces and corresponding norms will be used throughout the paper, see e.g., \cite{Ada1975}. In particular, $(\cdot,\cdot)$ and $\|\cdot \|$ denote $L^2(\Omega)$ inner product and the corresponding norm, respectively. $\V{H}^{k}$, where $k$ is an integer greater than zero, will denote the space of vector-valued functions each of whose $n$ components belong to $H^k$, the Sobolev space of real-valued functions with square integrable derivatives of order up to $k$ equipped with the usual norm $\|\cdot\|_k$. The dual space of $\V{H}^{1}_0(\Omega)$ will be denoted by $\V{H}^{-1}$ \DY{and the duality pairing between these two spaces is denoted by $\langle \cdot,\cdot \rangle$.} The norm in 
$\V{H}^{-1}$ is given by $\|\V{f} \|_{-1} = \langle \V{f}, (-\Delta )^{-1}\V{f} \rangle^{1/2}.$

The equivalent weak formulation of \eqref{eq:NSE1}-\eqref{eq:NSE3} reads as follows: $\forall (\V{v},q)\in(\V{X},Q) $,  find $(\V{u},p)\in (\V{X},Q)$ satisfying 
\begin{align}
a(\V{u}, \V{v})+c^*(\V{u},\V{u},\V{v})+b(p,\V{v}) & = \left<\V{f},\V{v} \right>, \label{eq:exact1}\\
b(q,\V{u}) & = 0, \label{eq:exact2}
\end{align}
where $\V{X}:= \V{H}_0^1(\Omega), \; Q:=L_0^2(\Omega)$ and 
\begin{align*}
a(\V{u}, \V{v}) & = \nu(\nabla \V{u}, \nabla \V{v}) + \gamma (\nabla \cdot \V{u}, \nabla \cdot \V{v}), \\
b(p, \V{v}) & = -(p,\nabla\cdot \V{v}), \\
c(\V{u},\V{v},\V{w}) &= ((\V{u}\cdot \nabla)\V{v},\V{w}), \\
c^* (\V{u},\V{v},\V{w}) &= c(\V{u},\V{v},\V{w})+\frac{1}{2}\left({(\nabla \cdot \V{u})}\V{v}, \V{w} \right).
\end{align*}
The following bound holds for all $\V{u},\V{v},\V{w} \in \V{X}$, see 
for instance
\cite{GirRav1986,Tem1979}:
\begin{eqnarray}
c^*(\V{u},\V{v},\V{w})\leq \mathcal{M} \|\nabla \V{u}\|\|\nabla \V{v}\|\|\nabla \V{w}\|, \label{eq:nnl}
\end{eqnarray}
for some $\mathcal{M}=\mathcal{O}(1)$. 
The assumption on $\Omega$ is sufficient to ensure that the inf--sup (or Ladyzhenskaya--Babu\v{s}ka--Brezzi, LBB) condition holds (see~\cite{BofBreFor2013,BoyFab2013,GirRav1986}). More precisely, there exists a constant $\beta > 0$, depending only on $\Omega$, such that
\begin{equation*}
\inf_{q \in Q} \;
\sup_{\V{v} \in \V{X}}
\frac{b(q,\V{v})}{\|q\|\, \|\nabla \V{v}\|}
\geq \beta.
\end{equation*} 
We also define the div-free subspace of $\V{X}$ as usual:
\begin{equation*}
    \V{V}: = \{ \V{v} \in \V{X}:\;  b(q,\V{v})=0 \; \forall q\in Q\},
\end{equation*}
which is obviously characterized by
\begin{equation*}
    \V{V} = \{ \V{v} \in \V{X}:\; \nabla \cdot  \V{v}=0 \; \text{in} \; \Omega\}.
\end{equation*}
For the operator\; ${\rm L}:= -\nu \Delta - \gamma \nabla \divergence:\; \V{X}\rightarrow \V{H}^{-1} $ that is associated with the bilinear form $a(\cdot,\cdot)$, we define:
\begin{equation}
    \label{eq:Lnorm}
    \|\V{g}\|_{\Lmap}: = \sqrt{\langle \Lmap \V{g},\V{g} \rangle} \text{ and } 
    \|\V{f}\|_{\Linvmap}: = \sqrt{\langle \Lmap^{-1} \V{f},\V{f} \rangle}.
\end{equation}
%
The following norm equivalences can be easily verified:
\begin{lemma}
The following inequalities hold:
\begin{equation}
\begin{aligned}
\forall \V{v} \in \V{X}, \, \sqrt{\nu}\|\nabla \V{v}\| \le \| \V{v} \|_\Lmap \le \sqrt{\nu+\gamma} \|\nabla \V{v}\|, \\
\text{ and } \forall \V{f} \in \V{H}^{-1}, \, 
 \frac{1}{\sqrt{\nu+\gamma}}\| \V{f} \|_{-1} \le \| \V{f} \|_\Linvmap \le \frac{1}{\sqrt{\nu}} \| \V{f} \|_{-1}.
\end{aligned}
    \label{eq:NormEquiv}
\end{equation}    
\end{lemma}
\begin{proof}
    The first equivalence easily follows by using the inequality $\|\divergence \V{v}\| \le \|\nabla \V{v}\|$ that holds for $\V{v} \in \V{X}$. The second equivalence follows from the first equivalence, and using the dual norm formulation of $\| \cdot \|_\Linvmap$:
    \begin{equation}
        \label{eq:DualNorm}
        \| \V{f} \|_\Linvmap = \sup\limits_{\V{0} \neq \V{v} \in \V{X}}\frac{\left\langle \V{f},\V{v} \right\rangle}{\| \V{v} \|_{\Lmap}}.
    \end{equation}
\end{proof}
Next, we state two preliminary lemmas on non-negative sequences that will be used in the sequel, taken from \cite{Calcolo2026}.
\begin{lemma}[Sequences converging to $0$]\label{lem:SeqConv0}
 Assume that $\{a_k\}_{k=1}^\infty, \{b_k\}_{k=1}^\infty, \{c_k\}_{k=1}^\infty$ are non-negative sequences of real numbers and $\exists \, \omega_i$, $\varepsilon_i$, $i=\overline{1,2}$, such that $0<\varepsilon_i \le \omega_i$ and 
 \begin{equation*}
  \omega_1 a_{k+1} + \omega_2 b_{k+1} + c_{k+1} \leq \left(\omega_1-\varepsilon_1 \right) a_k + \left(\omega_2 - \varepsilon_2 \right) b_{k} + c_{k}.
 \end{equation*}
Then $\exists \, C \ge 0$ such that 
\begin{align*}
 \lim\limits_{n \rightarrow \infty} \left(a_k,b_k,c_k\right) & = (0,0,C).
\end{align*}
\end{lemma}
\begin{lemma}[Contractivity of sequences converging to $0$]\label{ContractivitySequence}
Assume that $\{a_k\}_{k=1}^\infty, \{b_k\}_{k=1}^\infty, \{c_k\}_{k=1}^\infty$ are non-negative sequences of real numbers and $\exists \, \omega_i$, $i=\overline{1,3}$, $\varepsilon_i$, $i=\overline{1,2}$, such that $0<\varepsilon_i \le \omega_i$, $i=1,2$, 
 \begin{equation}
  \omega_1 a_{k+1} + \omega_2 b_{k+1} + \omega_3 c_{k+1} \leq \left(\omega_1-\varepsilon_1 \right) a_k + \left(\omega_2 - \varepsilon_2 \right) b_{k} + \omega_3 c_{k}
  \label{eq:Contract1}
 \end{equation}
 and 
 \begin{equation}
 c_{k} \le \tau_1 a_{k+1} + \tau_2 a_{k} + \tau_3 b_{k+1} \text{ for some positive } \tau_i, i=\overline{1,3}.
 \label{eq:Contract2}
 \end{equation}
Then there exists a sequence, which is a linear combination of $a_{k}, b_{k}, c_{k}$ that is contracting towards $0$.
\end{lemma}
We also recall the well-posedness result for \eqref{eq:NSE1}-\eqref{eq:NSE3} from \cite{Tem2001}:
\begin{lemma}\label{AprioriBds}
The system \eqref{eq:NSE1}-\eqref{eq:NSE3} always has a solution which satisfies
\begin{equation}
    \| \nabla \V{u} \| \le \frac{\|\V{f}\|_{-1}}{\nu} \; \text{ and } \;\| p\| \le \frac{1}{\beta}\left(2\|\V{f}\|_{-1}
    + \Mbd_0 \frac{\|\V{f}\|_{-1}^2}{\nu^2} \right),
\end{equation}
where $\beta>0$ is the inf-sup constant. Additionally, the solution is unique under a small data condition:
\begin{equation}
	\Lambda_0 := \nu^{-2} \Mbd_0 \|\V{f}\|_{-1} < 1 \label{eq:smalldataExact},    
\end{equation}
where 
\begin{equation}
\Mbd_0 := \sup\limits_{\V{u}, \V{v}, \V{w} \in \V{X}} \dfrac{(\V{u} \cdot \nabla \V{v}, \V{w})}{\|\nabla \V{u}\| \|\nabla \V{v}\| \| \nabla \V{w}\|}. \label{eq:mmm}
\end{equation}
\end{lemma}
\section{Steady Incremental Viscosity Splitting method}\label{Sec-Algo}
We present our new SIVS method for solving the system \eqref{eq:NSE1}-\eqref{eq:NSE3}.  In the unsteady IVS scheme \cite{Yak2023}, the splitting is performed with respect to the time derivative and the viscosity terms. Since $\V{u}_t$ is absent in our system, the splitting is performed with respect to the viscous terms enhanced by grad-div stabilization:
\begin{algorithm} \label{algo1}
For an initial guess $(\V{u}_0,p_0) = (\V{0},0)$, 
some positive parameter $\gamma>0$, 
and for $k = 1, 2, \ldots $, we compute the following until convergence: 
\begin{enumerate}
\item[{\bf Step 1:}] Given $\V{u}_{k-1} \in \V{V}$,  $p_{k-1} \in Q$, find $ \V{\tilde{u}}_{k} \in \V{X}$ solving
\begin{equation}\label{Iter-Split1}
a(\V{\tilde{u}}_{k},\V{v})
+ c^*(\V{u}_{k-1}, \V{\tilde{u}}_{k}, \V{v}) + b(p_{k-1},\V{v}) = <\V{f},\V{v} >,
\qquad \forall \, \V{v} \in \V{X}.
\end{equation}
\item[{\bf Step 2:}] Find $(\V{u}_{k}, p_{k}) \in \V{X} \times Q$ 
solving
\begin{equation}\label{Iter-Split2}
\begin{aligned}
a(\V{u}_{k} - \V{\tilde{u}}_{k}, \V{v}) + b(p_k-p_{k-1},\V{v}) & = 0, \qquad \forall \, \V{v} \in \V{X}, \\
b(q, \V{u}_{k}) & = 0, \qquad \forall \, q \in Q.
\end{aligned}
\end{equation}
\end{enumerate}
\end{algorithm}
Next, we establish the uniform boundedness and convergence results for our Algorithm \ref{algo1}.
\begin{theorem}[Uniform boundedness]\label{boundedness}
Assume that $(\Vt{u}_k, \V{u}_k, p_k) \in  (\V{X},\V{V},Q) $ is a solution of Algorithm \ref{algo1}. If 
\begin{equation}
	\Lambda := \nu^{-2} \Mbd \|\V{f}\|_{-1} < \frac{1}{\sqrt{2}} \label{eq:smalldata},    
\end{equation}
then $\| \nabla \Vt{u}_k\|, \| \nabla \V{u}_k\| , \| \nabla p_k\|_{-1}$ are uniformly bounded, and $$\displaystyle \V{u}_k \xrightarrow{\V{V}} \V{u}, \, \Vt{u}_k \xrightarrow{\V{X}} \V{u}, \text{ and } \, \nabla p_k \xrightarrow{\Linvmap} \nabla p \text{ as } k \rightarrow \infty,$$
where $(\V{u},p) $ is the unique solution of  \eqref{eq:NSE1}--\eqref{eq:NSE3}.
\end{theorem} 
%
%
\begin{proof}
First, let us note that the uniqueness of the solution of \eqref{eq:NSE1}--\eqref{eq:NSE3} follows from condition~\eqref{eq:smalldata}.
Defining the errors as $\V{e}_k=\V{u}-\V{u}_k$, $\Vt{e}_k=\V{u}-\Vt{u}_k$, $\delta_k=p-p_k$, and subtracting \eqref{Iter-Split1}-\eqref{Iter-Split2} from \eqref{eq:exact1}-\eqref{eq:exact2} we get the error equations:

\begin{equation}\label{eq:err1}
a(\V{\tilde{e}}_{k},\V{v})
+ c^*(\V{e}_{k-1}, \V{u}, \V{v}) + c^*(\V{u}_{k-1}, \V{\tilde{e}_k}, \V{v}) + b(\delta_{k-1},\V{v}) = 0,
\qquad \forall \, \V{v} \in \V{X},
\end{equation}
and
\begin{equation}\label{eq:err2}
\begin{aligned}
a(\V{e}_{k} - \V{\tilde{e}}_{k}, \V{v}) + b(\delta_k-\delta_{k-1},\V{v}) & = 0, \qquad \forall \, \V{v} \in \V{X}, \\
b(q, \V{e}_{k}) & = 0, \qquad \forall \, q \in Q.
\end{aligned}
\end{equation}
Letting $\V{v}=\Vt{e}_k$ in \eqref{eq:err1} and using the skew-symmetry of the trilinear form $c^*(\cdot,\cdot,\cdot)$, we have
\begin{equation}\label{eq:err3}   
\begin{aligned}
\nu \|\nabla \V{\tilde{e}}_{k} \|^2 
+ \gamma \| \nabla \cdot \V{\tilde{e}}_{k} \|^2 
+  b(\delta_{k-1}, \V{\tilde{e}}_k )
& = - c^*(\V{e}_{k-1}, \V{u}, \Vt{e}_k) \implies \\
\nu \|\nabla \V{\tilde{e}}_{k} \|^2 
+ \gamma \| \nabla \cdot \V{\tilde{e}}_{k} \|^2 
+ b(\delta_{k-1}, \V{\tilde{e}}_k - \V{e}_k)
& = - c^*(\V{e}_{k-1}, \V{u}, \Vt{e}_k) \\
& \le \Mbd \|\nabla \V{e}_{k-1} \|\, \|\nabla \V{u} \| \, \|\nabla \V{\tilde{e}}_{k} \| \\
& \le \nu \Lambda \|\nabla \V{e}_{k-1} \|\, \|\nabla \V{\tilde{e}}_{k} \| \\
& \le \frac{\nu}{4} \|\nabla \V{\tilde{e}}_{k} \|^2 + \nu \Lambda^2 \|\nabla \V{e}_{k-1} \|^2.
\end{aligned}
\end{equation}
To deal with $b(\delta_{k-1}, \V{\tilde{e}}_k - \V{e}_k)$ term, we first test \eqref{eq:err2} with $\V{v} = \V{e}_k \in \V{V}$ to obtain that
\begin{equation}\label{eq:err4}
\nu \left( \|\nabla \V{e}_{k} \|^2 - \|\nabla \V{\tilde{e}}_{k} \|^2 + \|\nabla (\V{e}_{k}  - \Vt{e}_{k}) \|^2 \right) + 
\gamma \left( \|\divergence \V{e}_{k} \|^2 - \|\divergence \V{\tilde{e}}_{k} \|^2 + \|\divergence (\V{e}_{k}  - \Vt{e}_{k}) \|^2 \right) = 0.
\end{equation}
Moreover, the first equation of \eqref{eq:err2} implies that $\Vt{e}_{k}-\V{e}_{k} = \Linvmap \nabla (\delta_{k}  - \delta_{k-1})$  in $\V{H}^{1}_0$, and that
\begin{equation}\label{eq:err5}
\|\nabla (\delta_{k}  - \delta_{k-1}) \|^2_{\Linvmap} = \|\V{e}_{k}  - \Vt{e}_{k} \|^2_{\Lmap}
= \nu \|\nabla(\V{e}_{k}  - \Vt{e}_{k}) \|^2 + 
\gamma \|\divergence(\V{e}_{k}  - \Vt{e}_{k}) \|^2.
\end{equation}
Thus, 
\begin{equation}
    \label{eq:err6}
    \begin{aligned}        
b(\delta_{k-1}, \V{\tilde{e}}_k - \V{e}_k) & = \left\langle \nabla \delta_{k-1},  \V{\tilde{e}}_k - \V{e}_k \right \rangle \\ 
& = \left\langle \nabla \delta_{k-1}, \Linvmap \nabla (\delta_{k}  - \delta_{k-1}) \right \rangle \\
& = \frac{1}{2} \left[\|\nabla \delta_{k} \|^2_{\Linvmap} - \|\nabla \delta_{k-1} \|^2_{\Linvmap} - \|\nabla (\delta_{k}  - \delta_{k-1}) \|^2_{\Linvmap} \right] \\
& = \frac{1}{2} \left[\|\nabla \delta_{k} \|^2_{\Linvmap} - \|\nabla \delta_{k-1} \|^2_{\Linvmap} \right]  - \frac{\nu}{2}  \|\nabla (\V{e}_{k}  - \Vt{e}_{k}) \|^2
-\frac{\gamma}{2} \|\divergence(\V{e}_{k}  - \Vt{e}_{k}) \|^2.
    \end{aligned}
\end{equation}
Combine \eqref{eq:err3} with the last identity to get
\begin{equation}\label{eq:err7}   
\begin{aligned}
\frac{3\nu}{4} \|\nabla \V{\tilde{e}}_{k} \|^2 
+ \gamma \| \nabla \cdot \V{\tilde{e}}_{k} \|^2 
& + \frac{1}{2} \left( \|\nabla \delta_{k} \|^2_{\Linvmap} - \|\nabla \delta_{k-1} \|^2_{\Linvmap} \right) \\
- \frac{\nu}{2}  \|\nabla (\V{e}_{k}  - \Vt{e}_{k}) \|^2
& -\frac{\gamma}{2} \|\divergence(\V{e}_{k}  - \Vt{e}_{k}) \|^2  \le \nu \Lambda^2 \|\nabla \V{e}_{k-1} \|^2.
\end{aligned}
\end{equation}
Now divide \eqref{eq:err4} by $2$ and add to \eqref{eq:err7} to obtain
\begin{equation}\label{eq:err8}   
\begin{aligned}
\frac{\nu}{2} \|\nabla \V{e}_{k} \|^2 +  \frac{\nu}{4} \|\nabla \V{\tilde{e}}_{k} \|^2 
+ \frac{\gamma}{2}\left( \| \nabla \cdot \V{\tilde{e}}_{k} \|^2 + \| \nabla \cdot \V{e}_{k} \|^2 \right)
+ \frac{1}{2} \|\nabla \delta_{k} \|^2_{\Linvmap}
& \le \frac{1}{2} \|\nabla \delta_{k-1} \|^2_{\Linvmap} + \nu \Lambda^2 \|\nabla \V{e}_{k-1} \|^2.
\end{aligned}
\end{equation}
Now, thanks to \eqref{eq:smalldata}, we can invoke the induction argument, to conclude that
\begin{equation}\label{eq:err9}   
\begin{aligned}
\frac{\nu}{2} \|\nabla \V{e}_{k} \|^2 +  \frac{\nu}{4} \|\nabla \V{\tilde{e}}_{k} \|^2 
+ \frac{\gamma}{2}\left( \| \nabla \cdot \V{\tilde{e}}_{k} \|^2 + \| \nabla \cdot \V{e}_{k} \|^2 \right)
+ \frac{1}{2} \|\nabla \delta_{k} \|^2_{\Linvmap}
& \le \frac{1}{2} \|\nabla \delta_{0} \|^2_{\Linvmap} + \nu \Lambda^2 \|\nabla \V{e}_{0} \|^2 \\
& = \frac{1}{2} \|\nabla p \|^2_{\Linvmap} + \nu \Lambda^2 \|\nabla \V{u} \|^2.
\end{aligned}
\end{equation}
The triangle inequality then implies the uniform boundedness of the solution:
\begin{equation}
\label{eq:err9}
\exists K_i>0, \, i=\overline{1,3}, \text{ with } \|\nabla \V{u}_k\| \leq K_1, \, \|\nabla \Vt{u}_k\| \leq K_2, \text{ and } \| \nabla p_k \|_{\Linvmap} \le K_3.     
\end{equation}
Moreover, Lemma~\ref{lem:SeqConv0} implies that 
\begin{equation}
    \label{eq:err9.5}
    \lim \limits_{k \rightarrow \infty} \|\nabla \V{e}_{k} \| = \lim \limits_{k \rightarrow \infty} \|\nabla \Vt{e}_{k} \|= \lim \limits_{k \rightarrow \infty} \|\divergence \V{e}_{k} \| = \lim \limits_{k \rightarrow \infty} \|\divergence \Vt{e}_{k} \| = 0,
\end{equation}
where we took
\begin{equation} \label{eq:err10}  
\begin{aligned}
a_k & = \|\nabla \V{e}_{k} \|^2, \, b_k = \frac{\nu}{4}\|\nabla \Vt{e}_{k} \|^2 + \frac{\gamma}{2} \left( \|\divergence \Vt{e}_{k} \|^2 + \|\divergence \V{e}_{k} \|^2 \right), \, c_k = \frac{1}{2}\|\nabla \delta_{k} \|^2_{\Linvmap}, \\
\omega_1 & = \dfrac{\nu}{2},  \omega_2=1, \\
\varepsilon_1 & = \dfrac{\nu}{2} - \nu \Lambda^2>0, \, \varepsilon_2 = 0.
\end{aligned}
\end{equation}
Then passing to the limit as $k \rightarrow \infty$ in \eqref{eq:err1} implies the convergence of $\|\nabla \delta_k\|_{\Linvmap}$ to zero as $k \rightarrow \infty$. 
\end{proof}
\begin{theorem}[Contractivity]\label{contractivity} 
Let $(\V{e}_k, \Vt{e}_k, \delta_k)$ be as in the Theorem~\ref{boundedness}. If $\Lambda < \dfrac{1}{\sqrt{2}} $ as in \eqref{eq:smalldata}, then there exists $\omega_j>0$, $j=\overline{1,3}$, and $\tau \in (0,1)$ such that
 \begin{equation} 
\psi_k: = \omega_1 \DY{\| \nabla \V{e}_k \|^2}
+ \omega_2 \left( \frac{\nu}{4}\|\nabla \Vt{e}_{k} \|^2 + \frac{\gamma}{2} \left( \|\divergence \Vt{e}_{k} \|^2 + \|\divergence \V{e}_{k} \|^2 \right) \right)
+ \omega_3 \| \nabla \delta_k \|_{\Linvmap}^2
\label{eq:ContractionThm1}
 \end{equation}
is a contractive sequence satisfying
 \begin{equation}
     \psi_{k+1} < \tau \psi_k.
     \label{eq:ContractionThm2}
 \end{equation}
\end{theorem}
\begin{proof}
The proof continues the previous proof by invoking the Lemma \ref{ContractivitySequence}. Note that \eqref{eq:err10} is equivalent to \eqref{eq:Contract1}.
However, applying the inf-sup condition in \eqref{eq:err1} gives that 
\begin{equation}
\label{eq:ContractionThm3}
  \begin{aligned}
   \|\nabla \delta_k\|_{-1} & \le \left( \nu + \Mbd K_1\right)  \| \nabla \Vt{e}_k \| 
   + \Mbd \|\nabla \V{u} \|  \DY{ \| \nabla \V{e}_{k-1} \|} 
   + \gamma \| \divergence \Vt{e}_k \|,
    \end{aligned}
\end{equation}
which is equivalent to \eqref{eq:Contract2}. The proof then directly follows from Lemma~\ref{ContractivitySequence} and the norm equivalence \eqref{eq:NormEquiv}. 
\end{proof}
\begin{remark}
    \begin{equation}
        \label{eq:Necas}
      \exists C > 0 \;  \forall q \in Q: \| \nabla  q \|_{-1} \ge C \| q\|,
    \end{equation}
we can conclude that $\|p_k\|$ is uniformly bounded \DY{in $L^2(\Omega)$}, and $\|\delta_k\| \rightarrow 0$ as $k \rightarrow \infty$ as well. 
\end{remark}
\subsection{Algebraic splitting viewpoint}
After performing a standard finite element discretization of Algorithm \ref{algo1}, we get that Step 1 is equivalent to
\begin{equation}
\label{eq:Yosida1}
A_{k-1} \Vt{U}_k = \V{F} - B^T \V{P}_{k-1},  
\end{equation}
and Step 2 is
\begin{align}
\widetilde{A} \left(\V{U}_k-\Vt{U}_k \right) + B^T \left(\V{P}_k - \V{P}_{k-1} \right) &= \V{0}, \label{eq:Yosida2}  \\
B \V{U}_k &= \V{0}. \label{eq:Yosida3}
\end{align}

In \eqref{eq:Yosida1}-\eqref{eq:Yosida3}, $\widetilde{A}$ is matrix corresponding to the bilinear form $a(\cdot,\cdot)$, $A_{k-1} = \widetilde{A} + C_{k-1}$ with $C_{k-1}$ being the matrix corresponding to $c^*(\Vt{u}_{k-1},\cdot,\cdot)$, and $B$ is matrix corresponding to the mixed form $b(\cdot,\cdot)$.

Eliminating $\V{U}_k$ in \eqref{eq:Yosida2}-\eqref{eq:Yosida3}, yields that
\begin{equation}
    \label{eq:Yosida4}
   B\Vt{U}_k - B\widetilde{A}^{-1} B^T \left(\V{P}_k - \V{P}_{k-1} \right) = \V{0}. 
\end{equation}
Then \eqref{eq:Yosida1} and \eqref{eq:Yosida4} can be written in a block form as
\begin{equation}\label{eq:Yosida5}
\begin{bmatrix} 
A_{k-1} & 0 \\ 
B &  -B\widetilde{A}^{-1} B^T 
\end{bmatrix} 
\begin{bmatrix} \Vt{U}_k \\
\V{P}_k - \V{P}_{k-1} 
\end{bmatrix} 
= 
\begin{bmatrix} \V{F} - B^T \V{P}_{k-1} 
\\ 
\V{0}
\end{bmatrix},
\end{equation}
with SPD Schur complement matrix $S=-B \widetilde{A}^{-1} B^T $. Since $S$ is spectrally equivalent to the pressure mass matrix \cite{BenGolLie2005}, it can be easily preconditioned. Moreover, $S$ and the associated preconditioner are assembled only once. 

Noting that \eqref{eq:Yosida2} also implies
$\Vt{U}_k = \V{U}_k + \widetilde{A}^{-1} B^T \left(\V{P}_k - \V{P}_{k-1} \right) \implies$
\begin{equation}\label{eq:Yosida6}
\begin{bmatrix} 
\Vt{U}_k \\ 
\V{P}_k - \V{P}_{k-1} 
\end{bmatrix} = 
\begin{bmatrix} 
I & \widetilde{A}^{-1} B^T \\
0 & I 
\end{bmatrix} 
\begin{bmatrix} \V{U}_k  \\ 
\V{P}_k - \V{P}_{k-1} 
\end{bmatrix}.
\end{equation}
Combining \eqref{eq:Yosida5} and \eqref{eq:Yosida6} yields
\begin{equation}\label{eq:Yosida7}
\begin{bmatrix} 
A_{k-1} & 0 \\ 
B &  -B\widetilde{A}^{-1} B^T 
\end{bmatrix}
\begin{bmatrix} 
I & \widetilde{A}^{-1} B^T \\
0 & I 
\end{bmatrix} 
\begin{bmatrix} \V{U}_k  \\ 
\V{P}_k - \V{P}_{k-1} 
\end{bmatrix}
= \begin{bmatrix} \V{F} - B^T \V{P}_{k-1} \\ \V{0} \end{bmatrix} \implies
\end{equation}
\begin{equation}\label{eq:Yosida8}
\begin{bmatrix} 
A_{k-1} & A_{k-1}\widetilde{A}^{-1} B^T \\ 
B &  0 
\end{bmatrix}
\begin{bmatrix} \V{U}_k  \\ 
\V{P}_k - \V{P}_{k-1} 
\end{bmatrix}
= \begin{bmatrix} \V{F} - B^T \V{P}_{k-1} \\ \V{0} \end{bmatrix}. 
\end{equation}
Thus, in the notation of \eqref{eq:PicardSystem2}, our Algorithm \ref{algo1} corresponds to picking $H_1=H_2 = \widetilde{A}^{-1}$. 
\section{Numerical experiments}\label{Sec-Simulation}
We present three numerical experiments to verify the theoretical findings and to demonstrate the effectiveness of our Algorithm \ref{algo1} using FreeFem++ \cite{Hec2012} software. The first test is a convergence test with a manufactured solution. Next, we test our scheme on a classical 2D lid-driven cavity flow problem at various Reynolds numbers. The last numerical test is on a classical 3D flow around a circular cylinder problem from \cite{SchTurDurKraRan1996}. We consider the $(\V{P}_2,P_1)$ Taylor-Hood finite element pair throughout all computations, and use direct solvers for 2D problems. For the 3D problem, Step 1 \eqref{eq:Yosida1} is solved via the FGMRES method with the Additive Schwarz Method preconditioner. The saddle point system in Step 2 \eqref{eq:Yosida2} is 
also solved using the FGMRES method, preconditioned by a block-multiplicative mAL preconditioner \cite{MOULIN2019718}. In all tests, the stopping criterion is taken as 
$$\max\left \lbrace \dfrac{\|\V{P}_k-\V{P}_{k-1}\|_{\ell_2}}{\| \V{P}_k \|_{\ell_2}} \right \rbrace \leq 10^{-6}.$$ 
%
\subsection{Convergence test}
We consider the following manufactured solution of \eqref{eq:NSE1}-\eqref{eq:NSE3}:
\begin{align}\label{eq:truesol1}
	\V{u}=\left(%
	\begin{array}{c}
		2 (x^2 - x)^2 (y^2 - y)(2y - 1) \\
	   -2 (y^2 - y)^2 (x^2 - x)(2x - 1)%
	\end{array}%
	\right),\quad
	p = (2x-1)(2y-1).
\end{align}
The right-hand side $\V{f}$ corresponding to \eqref{eq:truesol1} is computed to satisfy \eqref{eq:NSE1}. We choose $\nu=1$, and test on varying spatial mesh sizes $h$. The initial velocity and pressure fields are taken to be zero. 

Corresponding errors in various norms and convergence rates for the Algorithm \ref{algo1} are presented in Table \ref{tab:RateTable}. The expected rates of convergence are achieved for all quantities, and the number of iterations is uniform.
\begin{table}[H]
 \begin{center} 
 \caption{Errors and rates for $\gamma=1$, $\nu=1$, and varying mesh sizes.}\label{tab:RateTable} 
\begin{tabular}{|c|c|c|c|c|c|c|c|c|c|}
\hline
$h$ & $\|\V{e}\|$& rate & $\|\nabla \V{e}\|$ & rate & $\|\divergence \V{e}\|$ & rate & $\|\delta\|$ & rate & nb of iterations \\
\hline
\hline 
1.41e-1 & 2.29e-3 & - & 8.66e-2 & - & 2.62e-1 & - & 2.64e-2 & - & 9 \\
\hline
7.07e-2 & 1.54e-4 & 3.89 & 1.20e-2 & 2.85 & 6.89e-2 & 1.93 & 2.36e-3 & 3.48 & 9 \\
\hline
3.54e-2 & 9.85e-6 & 3.97 & 1.55e-3 & 2.95 & 1.75e-2 & 1.98 & 3.11e-4 & 2.92 & 9 \\
\hline
1.77e-2 & 6.20e-7 & 3.99 & 1.96e-4 & 2.98 & 4.39e-3 & 1.99 & 6.60e-5 & 2.24 & 9 \\
\hline
8.84e-3 & 3.90e-8 & 3.99 & 2.47e-5 & 2.99 & 1.10e-3 & 2.00 & 1.61e-5 & 2.03 & 9 \\
\hline
\end{tabular}
\end{center}
\end{table}
%
%
%
\subsection{2D lid driven cavity flow}\label{lidalg1}
In this subsection, we test our Algorithms on a well-known 2D lid-driven cavity flow problem. The computational domain is $\Omega = (0,1)^2$, where the top lid is moving in the positive $x$ direction with a unit speed. The boundary conditions are taken to be no-slip along the remaining walls. To avoid the irregularity of the solution at the upper corners, we consider a regularized initial data at the upper boundary due to \cite{Frutos2016}.
%
We run the code for five different values of Reynolds number, ${\rm Re} = 100, 400, 1000, 3200$, and $5000$. In all runs, we set $\gamma = 1$. Meshes with various resolutions were tested, and we only report the results obtained on the finest $128 \times 128$ uniform mesh. We initiated the ${\rm Re}=100$ simulation with zero, while the runs at higher ${\rm Re}$ are started from the converged solution of the previous ${\rm Re}$ simulations.
For numerical comparison, in Figure \ref{fig:Cavity1}, we plot the values of the velocity components at the centerlines through the domain against the reference values of Ghia et. al. \cite{Ghia82}.
\begin{figure}[!h]
\centering
\includegraphics[scale=0.24]{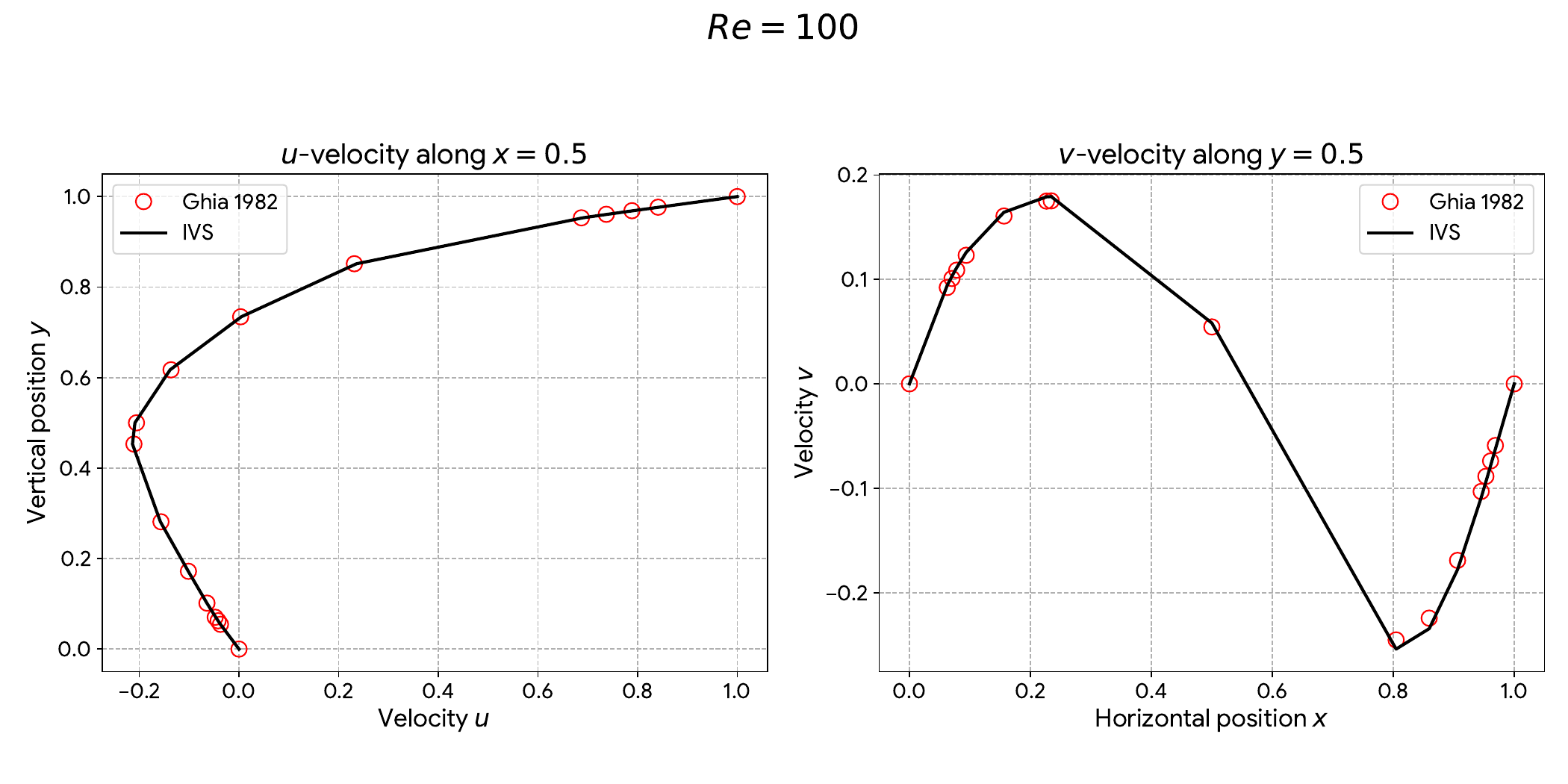} \\
\centering
\includegraphics[scale=0.24]{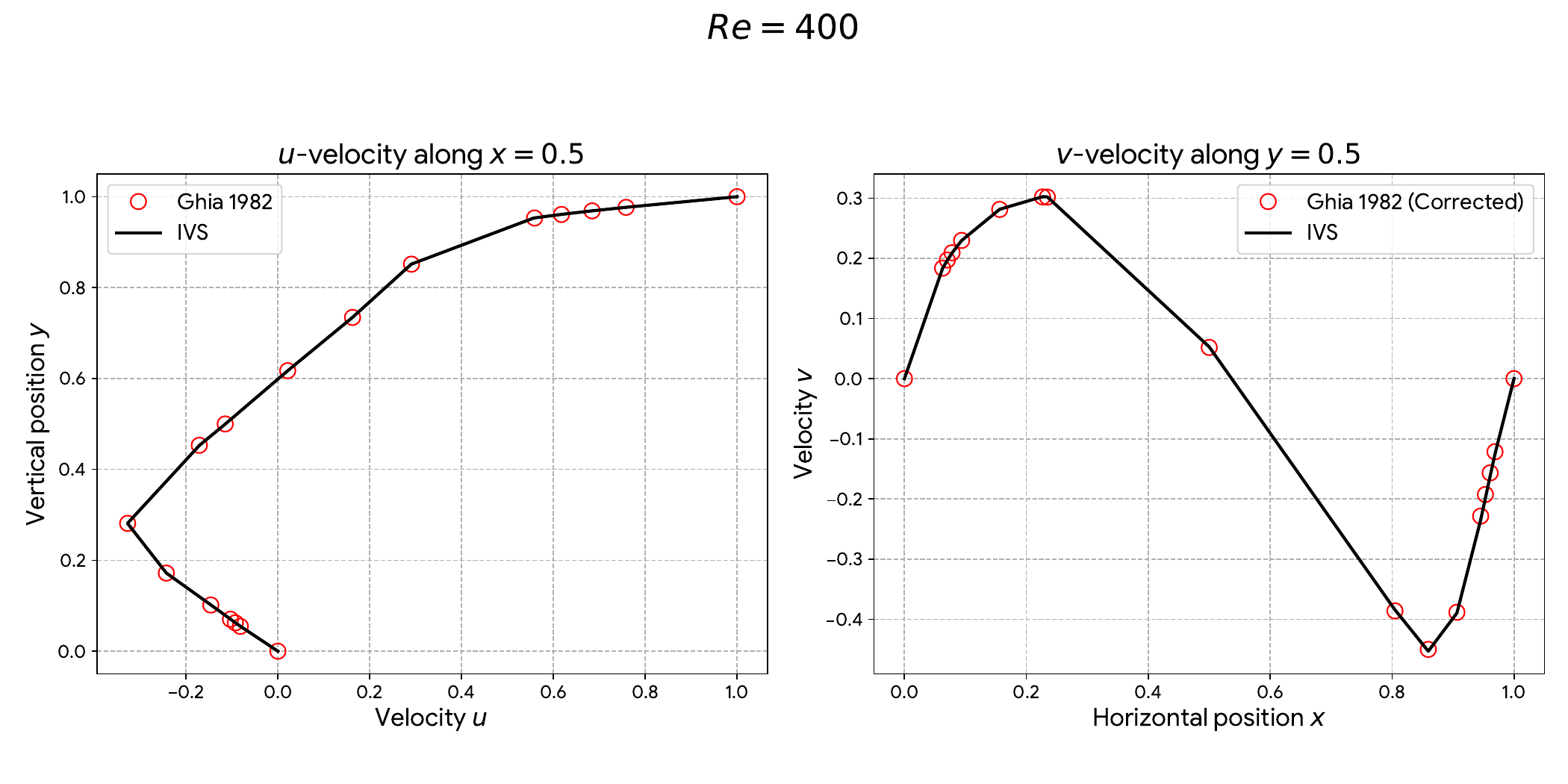} \\
\centering
\includegraphics[scale=0.24]{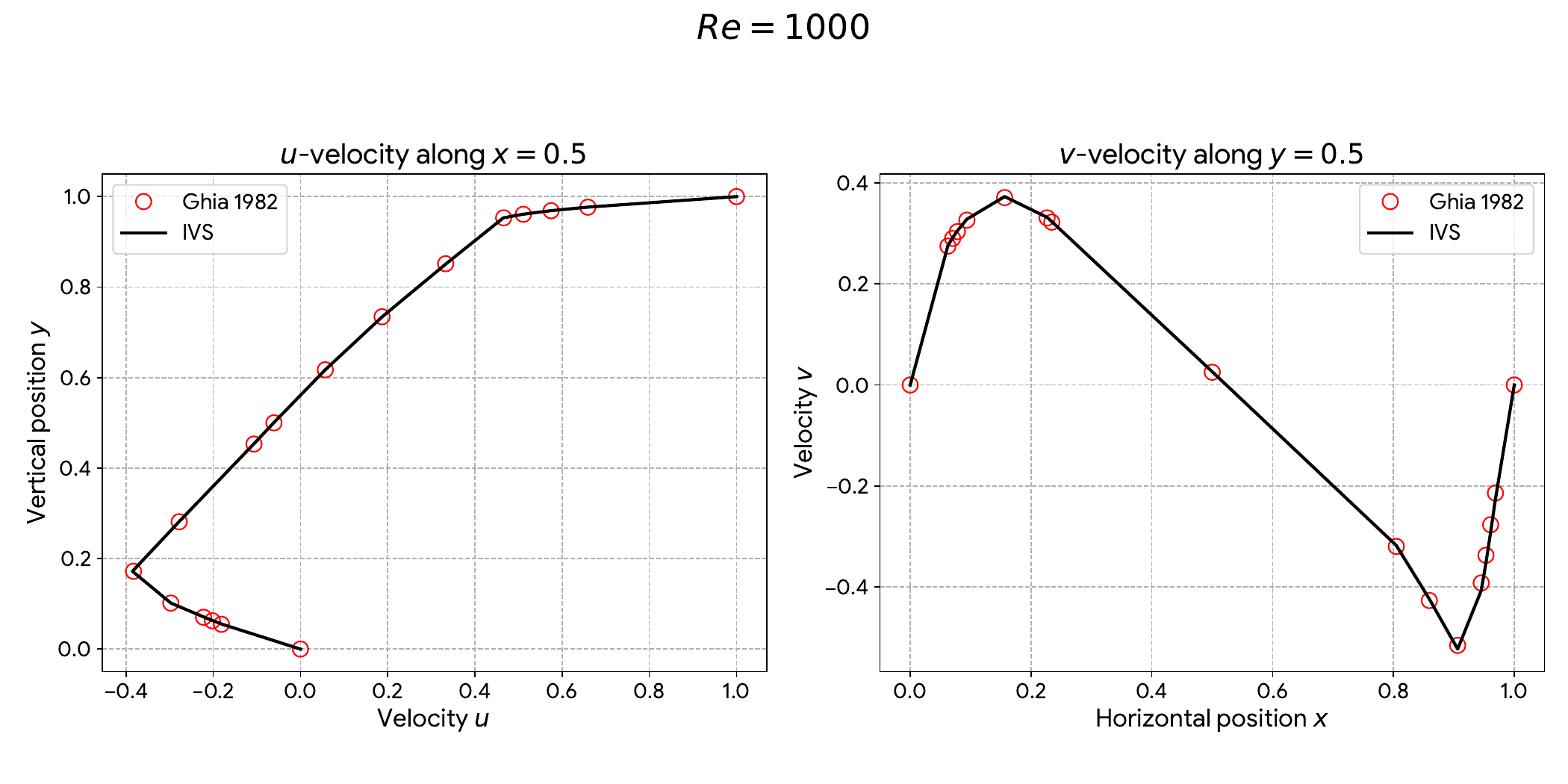} \\
\centering
\includegraphics[scale=0.24]{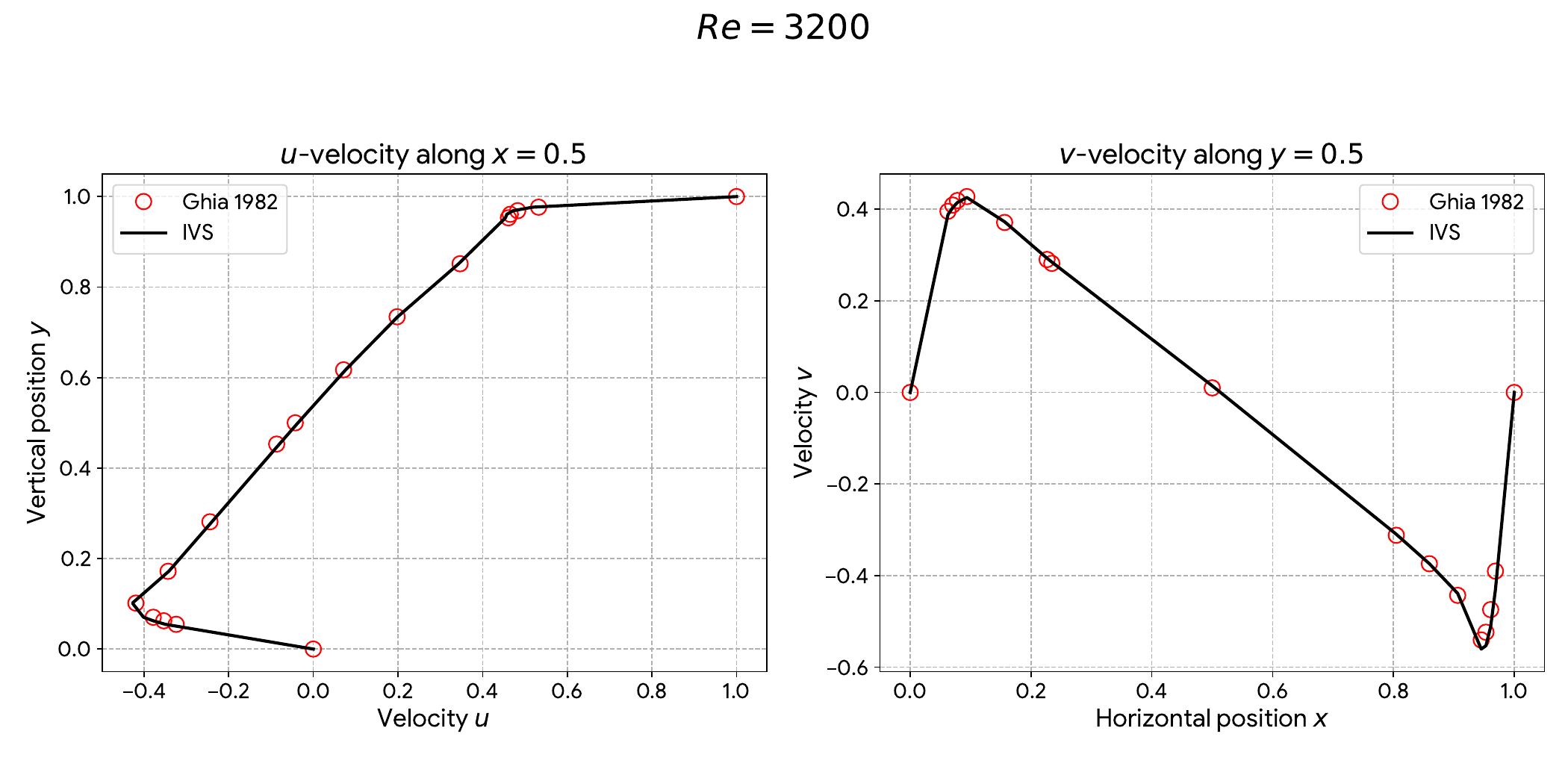} \\
\includegraphics[scale=0.24]{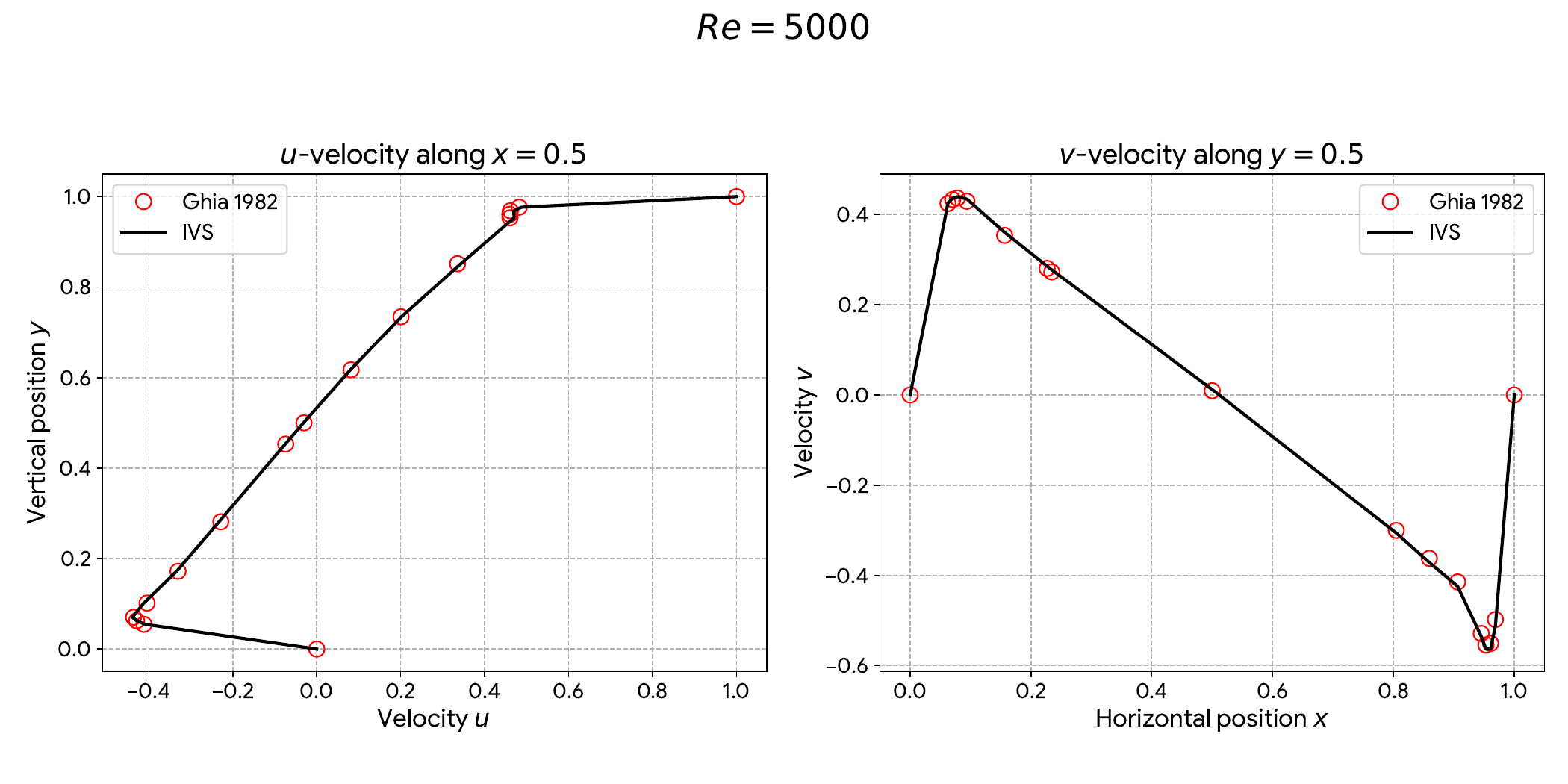}
\caption{Velocity components along domain centerlines}
\label{fig:Cavity1}
\end{figure}
The velocity streamlines superimposed on speed contours are shown in Figure \ref{fig:Cavity2}, which perfectly match the reference results \cite{Ghia82}. All plots demonstrate the correct formation of eddies in the corners. 

In Table \ref{tab:CavityIts}, we also present the number of iterations that were needed to attain convergence for our SIVS Algorithm and the IPY Algorithm of \cite{rebholz2019efficient}. Both schemes have similar, relatively robust performance in this regard. However, we expect that our scheme will consume less CPU time with iterative solvers, cf. \eqref{eq:Yosida7} and \eqref{eq:IPY}. 
\begin{table}[h]
\centering
\caption{Number of iterations required for convergence at various Reynolds numbers}
\label{tab:CavityIts}
\begin{tabular}{|c|c|c|}
\hline
${\rm Re}$ & \textbf{SIVS} & \textbf{IPY of \cite{rebholz2019efficient}} \\
\hline
100   &   17  & 12                      \\ \hline
400   &   30  & 28                      \\ \hline
1000  &   33  & 35                      \\ \hline
3200  &   62  & 57                      \\ \hline
5000  &   68  & 76                      \\ \hline
\end{tabular}
\end{table}
%
\begin{figure}[!htbp]
\centering
\begin{subfigure}{0.5\textwidth}
  \centering
\includegraphics[scale=0.29]{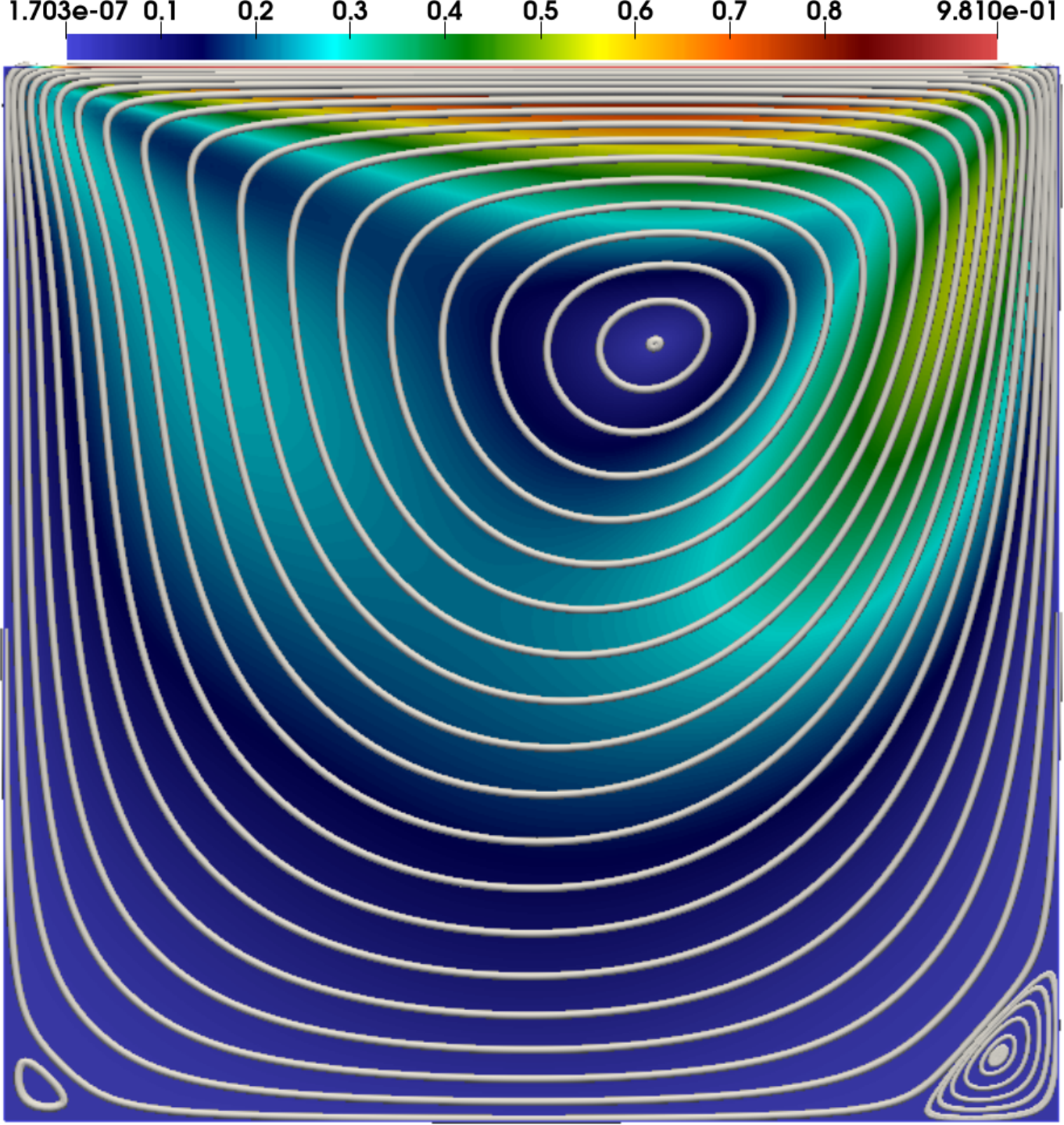}
\end{subfigure}%
\begin{subfigure}{0.5\textwidth}
  \centering
\includegraphics[scale=0.29]{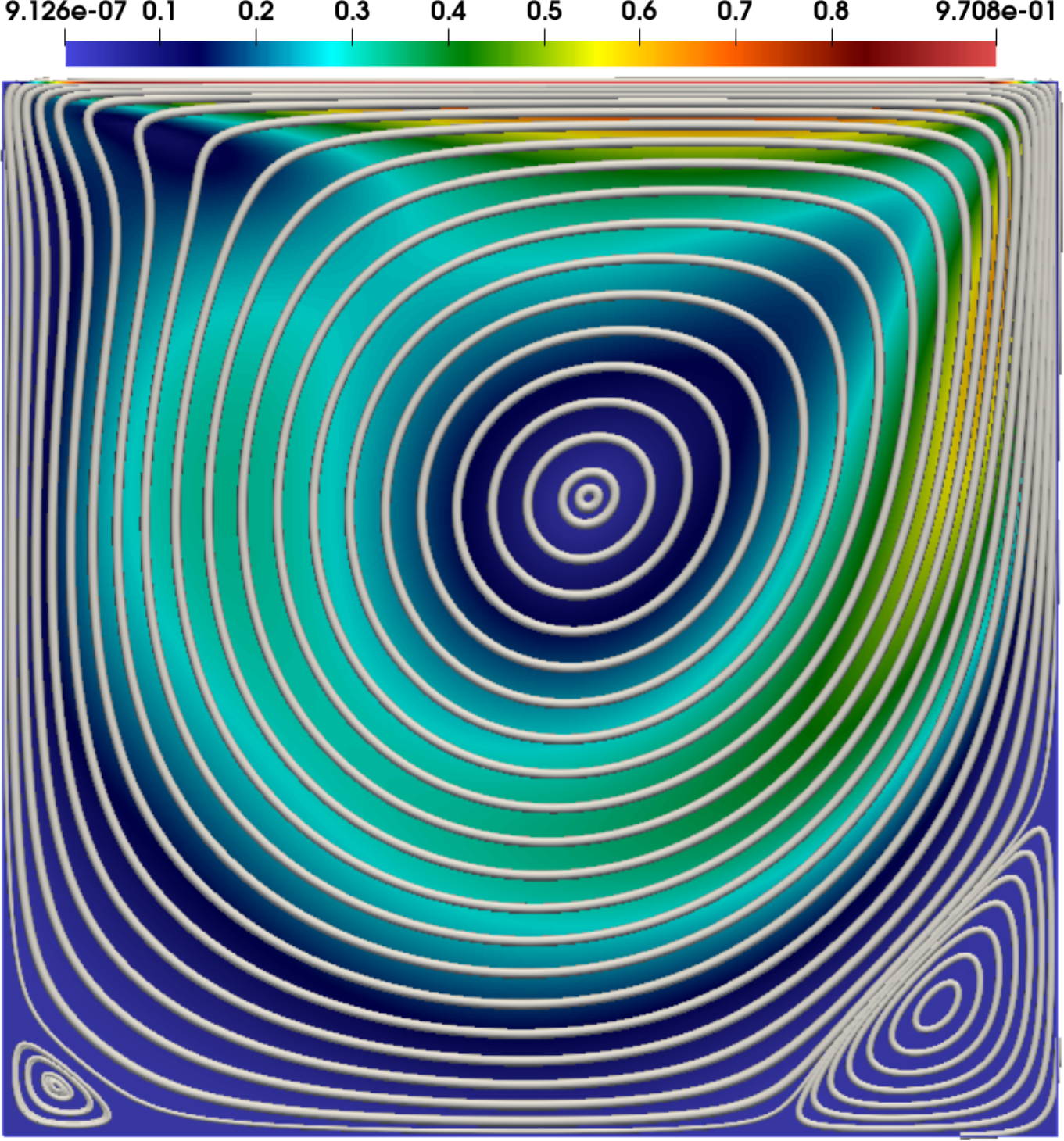}
\end{subfigure}
\centering
\begin{subfigure}{0.5\textwidth}
  \centering
\includegraphics[scale=0.29]{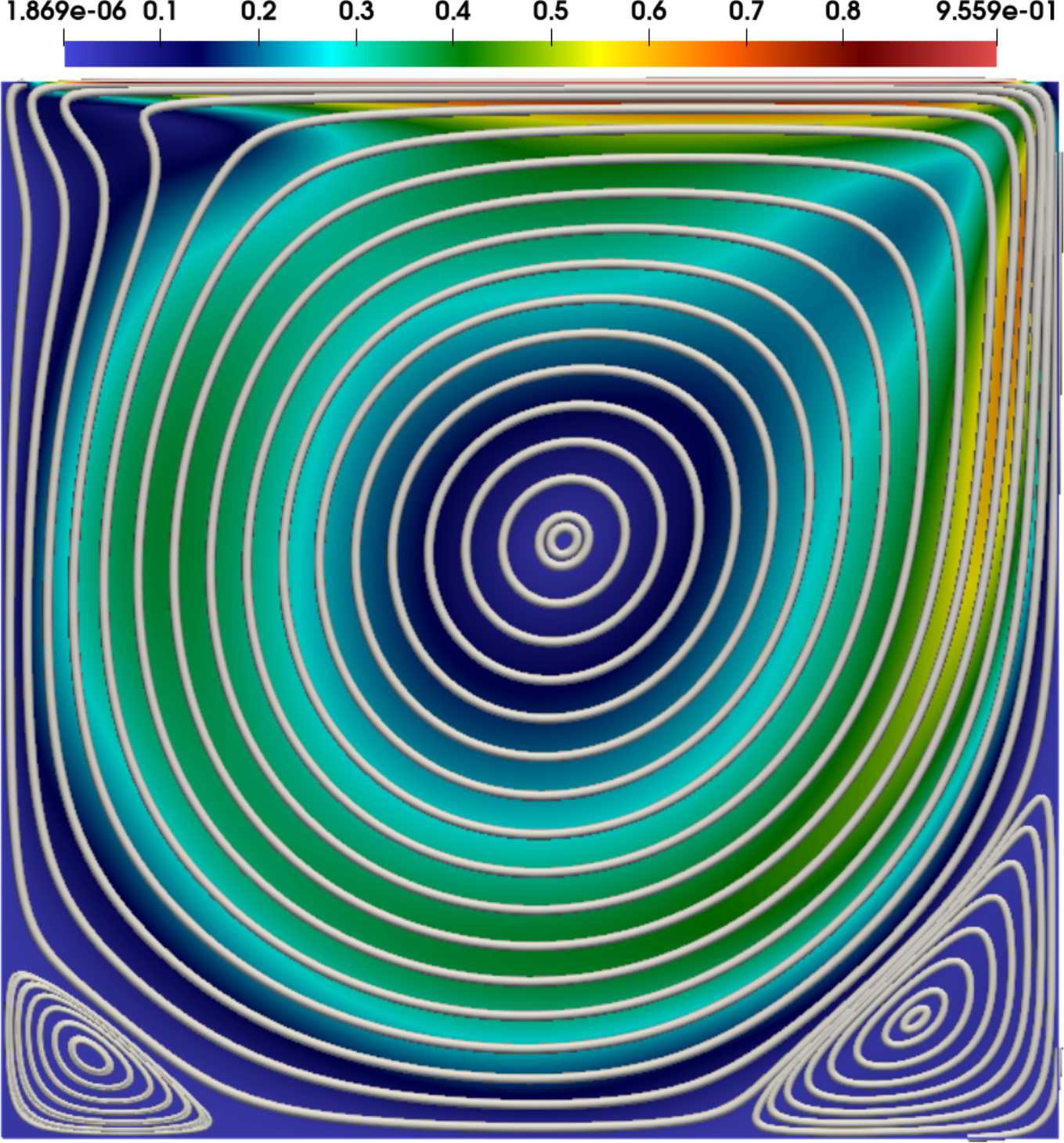}
\end{subfigure}%
\begin{subfigure}{0.5\textwidth}
  \centering
\includegraphics[scale=0.29]{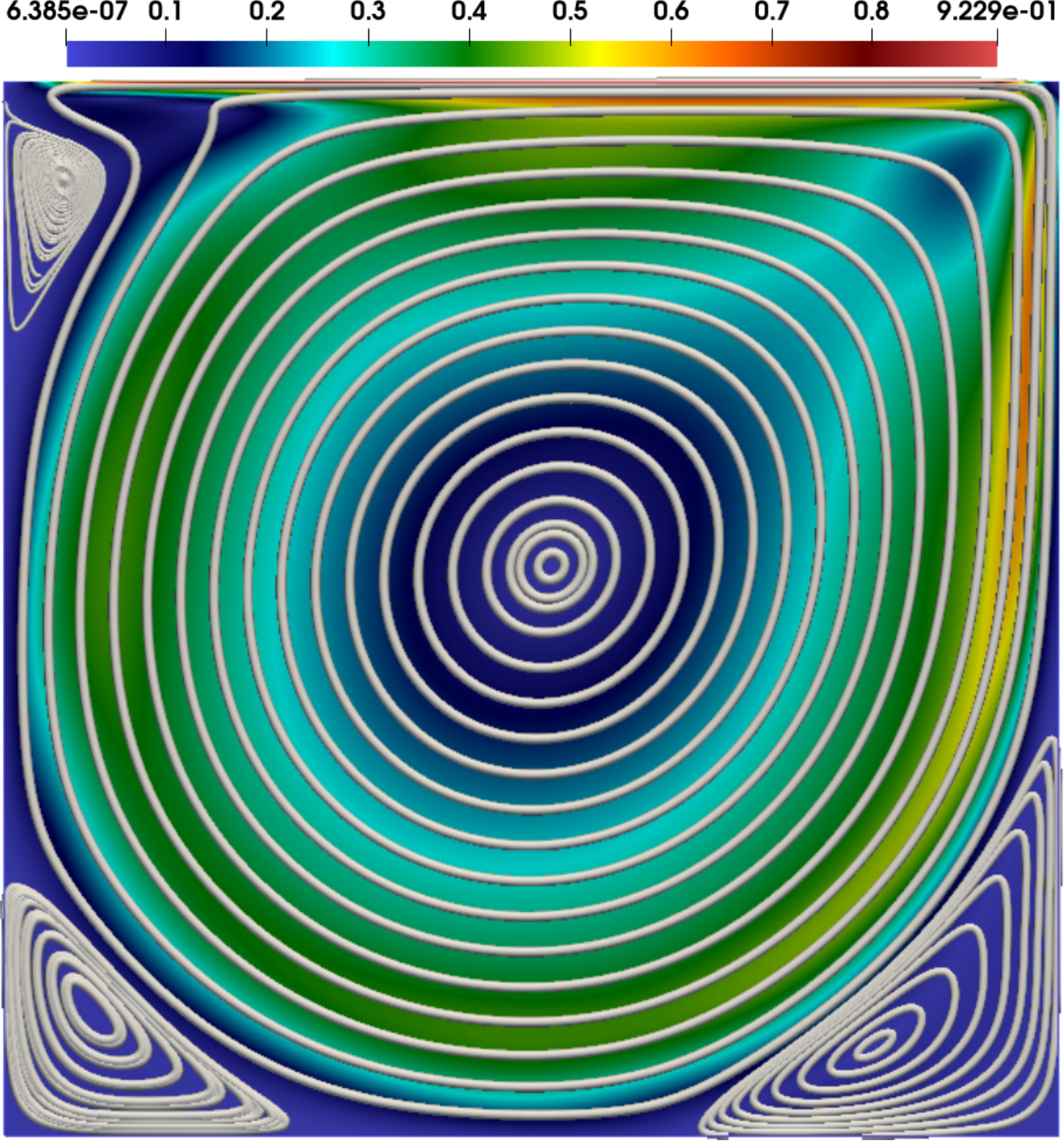}
\end{subfigure}
\begin{subfigure}{0.5\textwidth}
  \centering
\includegraphics[scale=0.29]{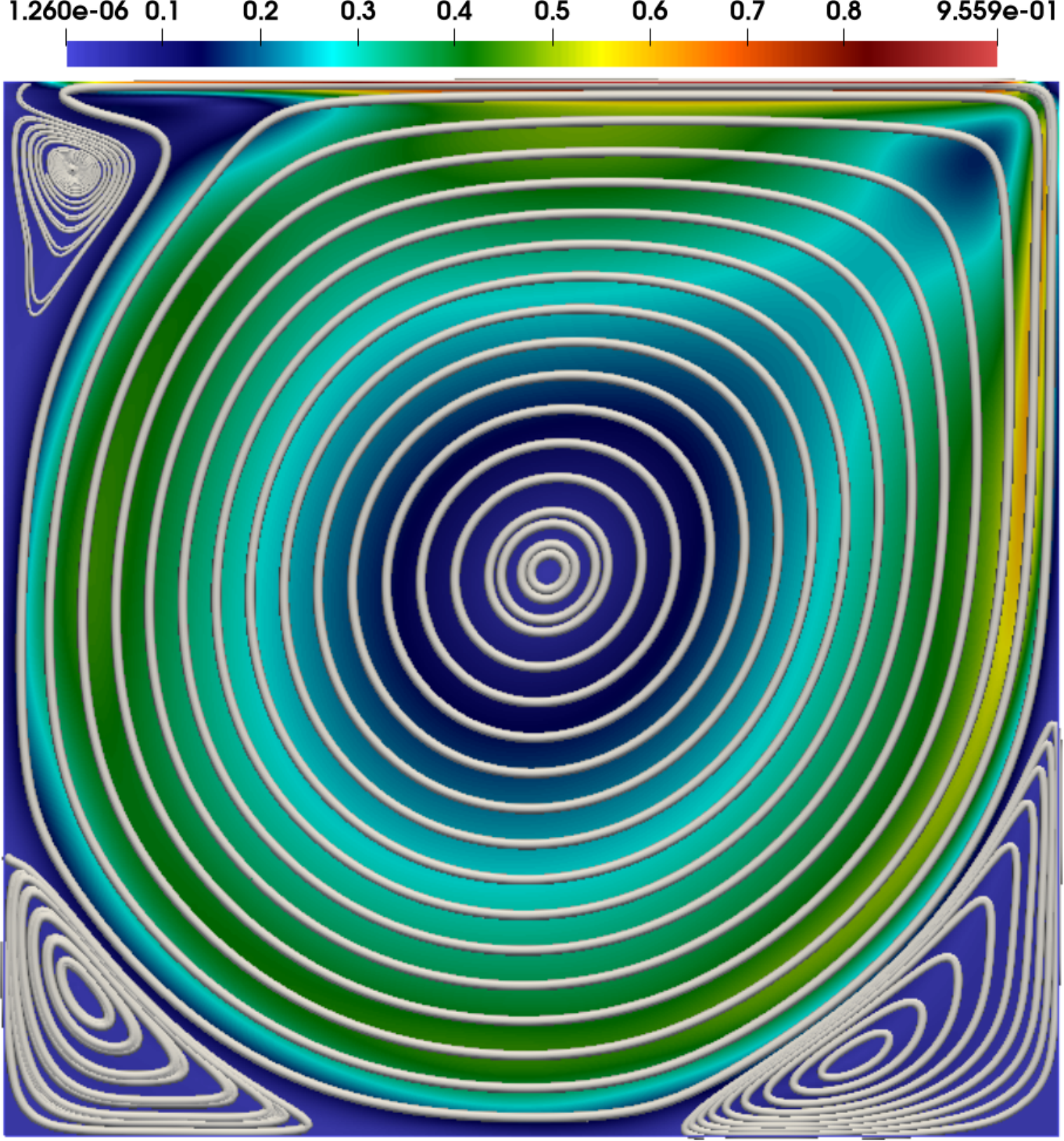}
\end{subfigure}
\caption{Velocity streamlines: Row 1 - ${\rm Re}=100, 400$; row 2 - ${\rm Re}=1000, 3200$; row 3 - ${\rm Re}=5000$}
\label{fig:Cavity2}
\end{figure}
%
\subsection{Influence of $\gamma$}
Here we examine the influence of $\gamma$ on the convergence rate of our Algorithm \ref{algo1} on a 2D lid-driven cavity test. For simplicity, we chose a coarser $64 \times 64$ mesh. We tested the range of values of $\gamma$. The iteration counts are given in Table \ref{tab:gammaTest}, which clearly indicates that a larger value of $\gamma$ accelerates the convergence to steady-state. However, as it is well-known \cite{BenGolLie2005}, choosing $\gamma \gg 1$ also deteriorates the conditioning of the linear systems. This is especially true for the system arising in Step 1, cf. \eqref{Iter-Split1}. On the other hand, thanks to the adapted right-hand side, the conditioning of the linear system arising from Step 2 of the SIVS scheme with respect to $\gamma$ might be mild, cf. \cite{https://doi.org/10.1002/num.22859,angot2012new}.
 \begin{table}[h!]
     \centering
\begin{tabular}{|c|c|}
\hline
$\gamma$ & \textbf{Nb of iterations} \\ \hline
1e-6 &  71  \\ \hline
1e-3 & 59   \\ \hline
1 &  12  \\ \hline
1e+2 &  10  \\ \hline
1e+3 &  8  \\ \hline
1e+6 &  3  \\ \hline
\end{tabular}
     \caption{Iterations for various values of $\gamma$ on a $64 \times 64$ mesh for 2D lid-driven cavity test}
     \label{tab:gammaTest}
 \end{table}

\subsection{3D flow past a circular cylinder}\label{sec-3DCylinder}
Our last test problem considers the 3D flow around a cylinder benchmark problem of \cite{SchTurDurKraRan1996}. The computational domain and the finite element mesh used in this study are shown in Figure~\ref{fig:3DTurekMesh}. 
\begin{figure}[!hbtp]
\centering
\includegraphics[scale=0.24]{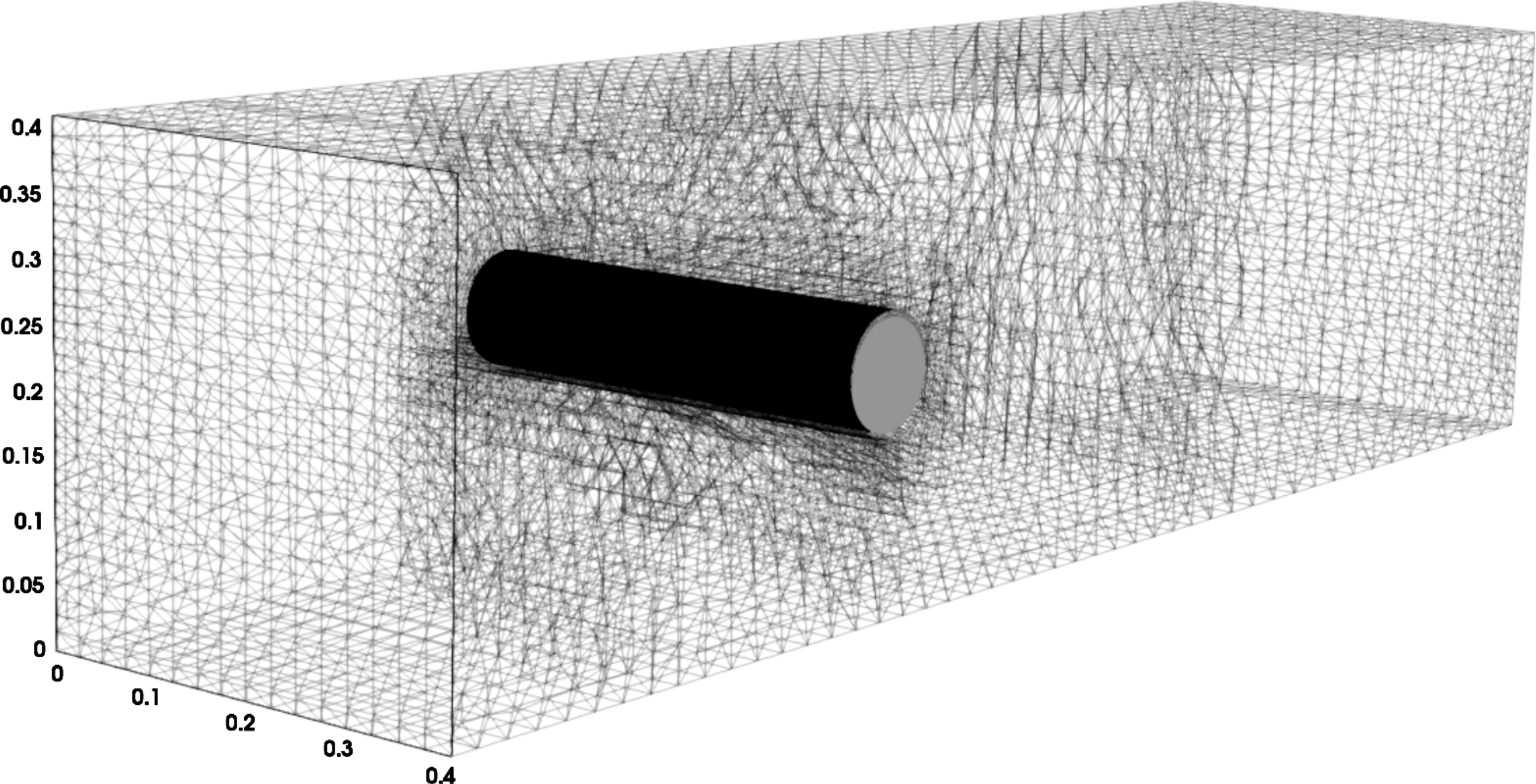} \\
\caption{The finite element mesh with a total of $694,430$ dofs and $0.0259 \le h \le 0.0612$.}
\label{fig:3DTurekMesh}
\end{figure}
Due to limited computational resources, we reduced the channel width from the original $2.5$ to $ 1.6$. Accordingly, $\Omega=1.6 \times 0.41 \times 0.41$, and the cylinder has a diameter of $D = 0.1$. 
The standard boundary conditions are imposed:
\begin{equation}\label{poiseuille3D}
\begin{array}{rcl}
\V{u} &=& \left(u_1:= 16 u_m  \,  \displaystyle \frac{yz(H-y)(H-z)}{H^4}\, ,\, u_2:= 0, \, u_3:= 0\right) \quad \text{on } x=0, \text{ (inlet) } \\
-\nu (\V{n} \cdot \nabla )\V{u} + p \V{n} &=& \V{0} \quad \text{on } x=1.6, \text{ (outlet) } \\
\V{u}(x,y) &=& \V{0} \quad \text{elsewhere}, 
\end{array}
\end{equation}
where the characteristic velocity is defined by $u_m = 0.45$.

Since we impose do-nothing outflow boundary conditions at the outlet with a simple geometry, we expect that this should have a negligible effect on the drag, lift coefficients, and the pressure drop about the cylinder. The external force is set to $\V{f} = \V{0}$, and the kinematic viscosity is $\nu = 10^{-3}$. The values of the coefficients reported in Table \ref{tab:Cylinder}, computed using a volume integrals approach \cite[Appendix D]{john2016finite}, have reasonable accuracy, even though the mesh is clearly underresolved. The convergence was achieved in $11$ nonlinear iterations.
\begin{table}[h]
\centering
\caption{Values of the drag, lift coefficients, and the pressure drop for the 3D cylinder problem}
\label{tab:Cylinder}
\begin{tabular}{|c|c|c|}
\hline
\textbf{Coefficient} & \textbf{SIVS} & \textbf{Reference intervals from  \cite{SchTurDurKraRan1996}} \\
\hline
drag   &   6.11785  & [6.05, 6.25]                      \\ \hline
lift   &   0.008738  & [0.008,0.01]                      \\ \hline
$\Delta p$  &   0.1744  & [0.165,0.175]                      \\ \hline
\end{tabular}
\end{table}
\begin{figure}[!h]
\centering
\includegraphics[scale=0.24]{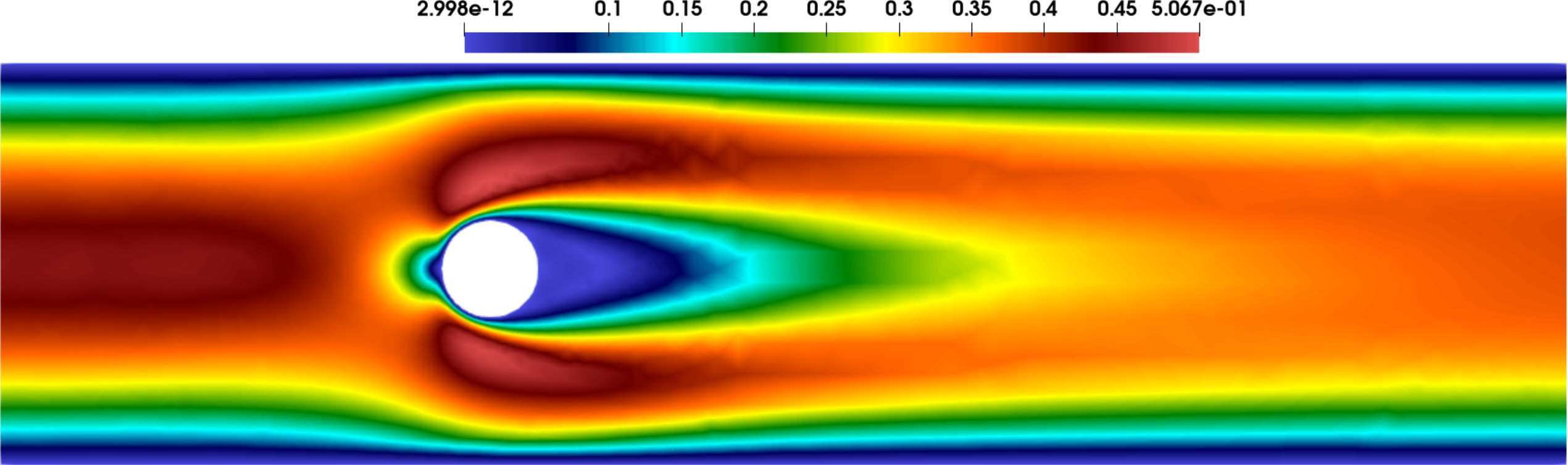} \\
\centering
\includegraphics[scale=0.24]{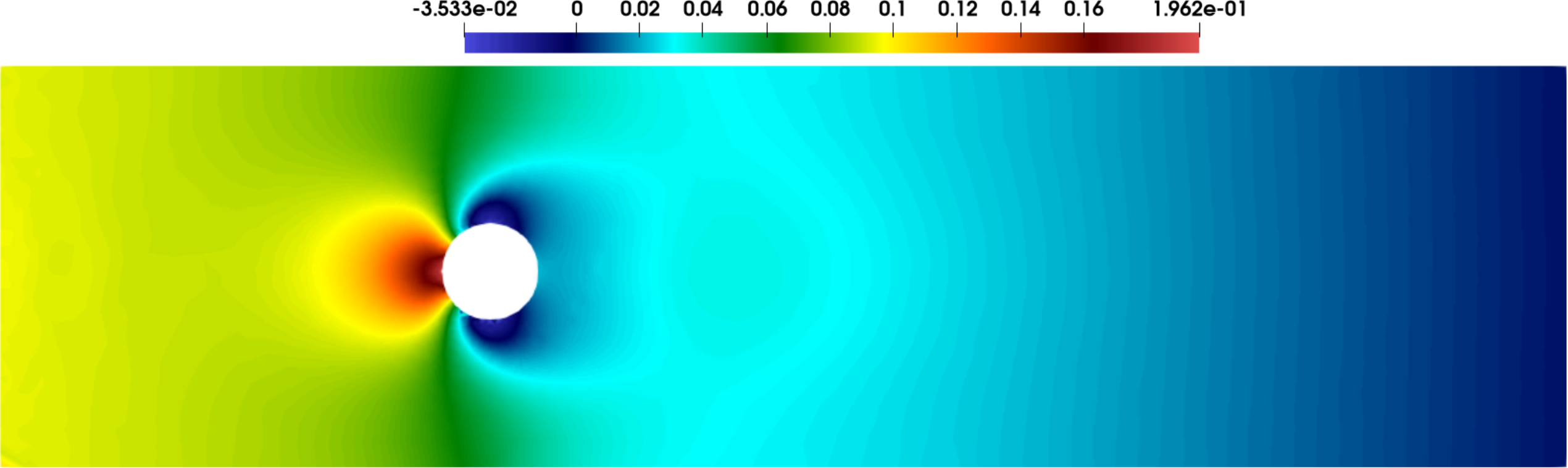} 
\caption{Finite element mesh, speed and pressure contours at mid $z$-plane.}
\label{fig:Cylinder}
\end{figure}
\section{Conclusions}\label{Sec-Conclusion}
We presented a new efficient scheme for solving the steady Navier-Stokes equations with grad-div stabilization. We proved boundedness and convergence under a reasonably smallness assumption on the data.
The algebraic splitting interpretation of the SIVS scheme shows that our scheme results in significantly simpler and cheaper linear systems compared to standard Picard linearizations, making it suitable for large 3D problems.

Numerically, the SIVS scheme demonstrated expected convergence rates and reproduced standard benchmarks accurately, including the 2D lid-driven cavity (up to ${\rm Re} = 5000$) and a 3D flow past a circular cylinder, while maintaining low iteration counts.  Future work will focus on extending this approach to steady MHD and other multiphysics problems.

\bibliographystyle{abbrv}
\bibliography{references}
\end{document}